\documentclass[12]{amsart}

\usepackage{amssymb}
\allowdisplaybreaks[1]

\numberwithin{equation}{section}
\newtheorem{theorem}{Theorem}[section]
\newtheorem{lemma}[theorem]{Lemma}
\newtheorem{proposition}[theorem]{Proposition}
\newtheorem{corollary}[theorem]{Corollary}

\newtheorem{question}[theorem]{Question}
\newtheorem{questions}[theorem]{Questions}

\theoremstyle{definition}
\newtheorem{definition}[theorem]{Definition}

\theoremstyle{remark}
\newtheorem{remark}[theorem]{Remark}

\newtheorem{convention}[theorem]{Convention}

\newtheorem{notation}[theorem]{Notation}

\newcommand{\Gcal}{{\mathcal{G}}}
\newcommand{\Lcal}{{\mathcal{L}}}

\newcommand{\Rcal}{{\mathcal{R}}}

\newcommand{\id}{\mathord{\mathrm{id}}}
\newcommand{\mpl}{\mathop{\mathrm{mpl}}}
\newcommand{\Sym}{\mathop{\mathrm{Sym}}}

\begin{document}
\title{Braces and symmetric groups with special conditions}
\author{Ferran Ced\'{o}, Tatiana Gateva-Ivanova, Agata Smoktunowicz}

\keywords{Yang--Baxter Equation, set-theoretic solutions, brace,
braided group, Jacobson radical.\\
* The first author was partially supported by the grant  MINECO MTM2014-53644-P\\
** The second author was partially supported by Grant I 02/18
 of the Bulgarian National Science Fund, by ICTP, Trieste,  and by Max-Planck Institute for Mathematics, Bonn\\
*** The third author was supported by ERC Advanced grant 320974. }

\date{\today}

\begin{abstract}  We study symmetric groups and left braces satisfying special conditions,
or identities. We are particularly interested in the impact of
conditions like \textbf{Raut} and \textbf{lri} on the properties of
the symmetric group and its associated brace. We show that the symmetric group $G=G(X,r)$
associated to a nontrivial solution $(X,r)$ has multipermutation
level $2$ if and only if $G$ satisfies  \textbf{lri}.
In the special case of a two-sided brace we express each of the conditions \textbf{lri} and \textbf{Raut} as identities on the
associated radical ring $G_*$.
 We apply these to construct examples of two-sided braces satisfying some prescribed
conditions. In particular we construct a finite two-sided brace with
condition \textbf{Raut} which does not satisfy \textbf{lri}. (It is known that condition \textbf{lri} always implies \textbf{Raut}).
We show that a finitely
generated two-sided brace which  satisfies \textbf{lri} has a
finite multipermutation level which is bounded by the number of its
generators.
\end{abstract}
\maketitle

\setcounter{tocdepth}{1}
\tableofcontents

\section{Preliminaries}
\label{Preliminaries}

A quadratic set is a pair $(X,r)$, where $X$ is a non-empty set and
$r\colon X\times X\rightarrow X\times X$ is a bijective map. Recall
that $(X,r)$ is involutive if $r^2=\id_{X^2}$.
The image of $(x,y)$ under $r$ is presented as
$r(x,y)=({}^xy,x^{y}).$
Consider the maps
$\mathcal{L}_x,\mathcal{R}_x\colon X\rightarrow X$ defined by
$$\mathcal{L}_x(y)={^x}y\quad\mbox{and}\quad \mathcal{R}_x(y)=y^x,$$
for all $x,y\in X$. The quadratic set $(X,r)$ is non-degenerate if
$\mathcal{L}_x$ and $\mathcal{R}_x$ are bijective for all $x\in X$.
The map $r$ is \emph{a set-theoretic solution of the
Yang--Baxter equation} (YBE) if  the braid relation
\[r^{12}r^{23}r^{12} = r^{23}r^{12}r^{23}\]
holds in $X\times X\times X,$  where  $r^{12} = r\times\id_X$, and
$r^{23}=\id_X\times r$.  In this
case $(X,r)$ is called \emph{a braided set}.
A braided set $(X,r)$ with $r$ involutive is called \emph{a
symmetric set}.

\begin{convention} In this paper "\emph{a solution}"  means
"\emph{an involutive non-degenerate set-theoretic solution of the
YBE }", or equivalently, "\emph{a non-degenerate symmetric set}" $(X,r)$.
\end{convention}
\emph{A left brace} is a triple $(G,+,\cdot)$, where $G$ is a set,
$+$ and $\cdot$ are two binary operations, such that $(G,+)$ is an
abelian group, $(G,\cdot)$ is a group and
\begin{eqnarray}\label{lbrace}
&& a\cdot (b+c)+a=a\cdot b+a\cdot c,
\end{eqnarray}
for all $a,b,c\in G$.  The group $(G,+)$ is the additive group of
the left brace and $(G,\cdot)$ is its multiplicative group. \emph{A
right brace} is defined similarly, but replacing property
(\ref{lbrace}) by $(b+c)\cdot a+a=b\cdot a+ c\cdot a$. If
$(G,+,\cdot)$ is both a left and a right brace (for the same
operations), then it is called \emph{a two-sided brace}.

It is known that if $(G,+,\cdot)$ is a left brace,  and  $0$ and
$e$, respectively,  denote the neutral elements with respect to the
two operations "$+$" and "$\cdot$" in $G$, then $0 = e$.

In any left brace $(G,+,\cdot)$ one defines the operation $*$ by the
rule:
\begin{equation}
\label{eqleftbrace*}
a*b=a\cdot b-a-b,  \; a, b \in G.
\end{equation}
 It is known and easy to check that $*$ is left
distributive with respect to the sum $+$. In general $*$ is not
right distributive, nor associative, but it satisfies the following
condition
\begin{eqnarray}\label{rumpdef}
&&(a*b+a+b)*c=a*(b*c)+a*c+b*c, \; \forall a,b,c \in G,
\end{eqnarray}
see the original definition of right brace of
Rump  \cite[Definition~2]{Ru07}. It is also known that
$(G,+,\cdot)$ is a two-sided brace if and only if $(G,+,*)$ is a
Jacobson radical ring.

Takeuchi introduced the notions of a braided group and a symmetric
group as the group versions of a braided set and a symmetric set,
respectively, \cite{Takeuchi}. We recall the definitions.

\emph{A braided group}  is a pair  $(G, r)$, where $G$ is a group and
$r: G\times G \longrightarrow G\times G, \quad
r(a, b) = ({}^ab, a^b)$ is a  bijective map satisfying the
following conditions
\[
\begin{array}{llll}
\textbf{ML0}:\quad&{}^a1 =1, \; {}^1u =u, \quad &\textbf{MR0}:\quad & 1^u =1,
\; a^1 =a,
\\
\textbf{ML1}: &{}^{ab}u ={}^a{({}^bu)},\quad &\textbf{MR1}:& a^{uv} =(a^u)^v,
\\
\textbf{ML2}:&{}^a{(u.v)} =({}^au)({}^{a^u}v),\quad &\textbf{MR2}:& (a.b)^u
=(a^{{}^bu})(b^u), \\
\end{array}
\]
and the compatibility condition
\[ \mathbf{M3}:\quad uv = ({}^uv).(u^v),\]
for all $a,b,u,v\in G$.
For each braided group $(G, r)$ the map $r$ is  \emph{a braiding operator}, so $(G, r)$ is a braided set, see \cite{LYZ}, see also \cite{Takeuchi}.

\emph{A symmetric group} is a braided group $(G, r)$ with an involutive braiding operator $r$.
 Each symmetric group $(G,r)$ is a nondegenerate symmetric set, that is $(G,r)$ is a solution.

It was proven by the second author that symmetric groups and left
braces are equivalent structures, see \cite{Tatyana} Theorem 3.6.
More precisely, the following hold.

(i) Every symmetric group $(G, r)$ has a canonically associated structure of a left brace $(G, +, \cdot)$, where the
 operation "$+$" on $G$ is defined via
 \begin{equation}
\label{eqoperation1}
 a + b := a\cdot({}^{a^{-1}}b), \;\text{or equivalently},\; a+{}^ab = a\cdot b, \; a,b \in G.
 \end{equation}

(ii)  Conversely, every left brace $(G, +, \cdot)$ has a canonically associated structure of a symmetric
group $(G,r)$, that is
a group with a braiding operator  $r: G\times G \longrightarrow G\times G$, $r(a,b) := ({}^ab, a^b),$  with left and right actions of $G$ upon itself given by the formulae
  \begin{equation}
\label{eqoperation2}
 {}^ab:=a\cdot b-a=a*b+b, \quad\quad a^{b}:={}^{({}^ab)^{-1}}a, \; \; \forall\; a, b \in G.
  \end{equation}
 Moreover, the following condition holds in $G$
\[\mathbf{Laut:}\quad ^{a}(b+c)={^{a}}b+{^{a}}c,\quad\forall a,b,c\in
G.\]
By convention a symmetric groups $(G, r)$ is always considered
together with the associated left brace $(G, +, \cdot)$ and vice
versa.

For each solution $(X,r)$ of the YBE
Etingof, Schedler
and Soloviev introduced in \cite{ESS} two groups:
the structure
group $G=G(X,r)$ and the permutation group
$\mathcal{G}=\mathcal{G}(X,r)$. The
structure group $G$ is generated by $X$ and has quadratic defining relations
$xy={^x}yx^y$, for all $x,y\in X$. (The group $G(X,r)$ is also called the YB-group of $(X,r)$).
The set $X$ is embedded in $G$.
The group $G$ acts on the left (and on the right) on the set $X$, so the assignment $x\mapsto
\mathcal{L}_x$ extends to a group homomorphism
$\Lcal: G(X,r) \longrightarrow \Sym(X),  \; a \mapsto  \Lcal_a \in Sym(X)$, where $\Lcal_a(x) = {}^ax$.
By definition the permutation group $\Gcal=\Gcal(X,r)$ is the image $\Lcal(G(X,r))$ of $G$.
The
group $\Gcal$ is generated by the set $\{\Lcal_x \mid x \in X\}$.
It is known, see \cite{LYZ}, that
 there is unique braiding operator
$r_G: G \times G \longrightarrow G \times G $, such that the restriction of $r_G$ on $X\times X$ is exactly the map $r$. We call $(G, r_G)$ \emph{the symmetric group associated to} $(X,r)$.
Moreover, the epimorphism $\Lcal:G(X,r) \longrightarrow \Gcal(X,r)$ is a braiding preserving map which induces a canonical structure of a symmetric group $(\Gcal, r_{\Gcal})$, see
\cite{Tatyana}  (or \cite{CJO14} for the equivalent version in the
language of left braces).

An ideal of a left brace $(G,+,\cdot)$ is a normal subgroup $I$ of
its multiplicative group which is invariant with respect to the left action of $G$ upon itself,
i.e. $^{a}b\in I$ for all $a\in G$ and
all $b\in I$.  It is known that every ideal $I$ of $(G,+,\cdot)$ is
a subgroup of its additive group, and is invariant with respect to the right action of $G$.

Each left brace $(G, +, \cdot)$ has several invariant decreasing
chains of subsets.

The  series $G^{(n)}$, introduced by Rump, \cite{Ru07}, consists  of
ideals of $G$:
\begin{equation}
\label{Rideals1} G=G^{(1)} \supseteq G^{(2)}\supseteq
G^{(3)}\supseteq\cdots, \; \text{where}\;   \; G^{(n+1)}=G^{(n)}*G,
n \geq 1.
\end{equation}
The second series, $G^{n}$, \cite{Ru07}, is defined as
\begin{equation}
\label{Rumpchain2} G= G^{1} \supseteq G^{2}\supseteq
G^{3}\supseteq\cdots, \; \text{where}\;   G^{n+1}=G*G^{n},  \;n \geq
1.
\end{equation}

Recall the following definition.
\begin{definition} \cite{GIM08}
\label{lri&cl} Let $(X,r)$ be a quadratic set.
\begin{enumerate}
\item  The following are called \emph{cyclic conditions on}  $X$.
\[\begin{array}{lclc}
 {\rm\bf cl1:}\quad&{}^{(y^x)}x= {}^yx, \quad\text{for all}\; x,y \in
 X;
 \quad\quad&{\rm\bf cr1:}\quad &x^{({}^xy)}= x^y, \quad\text{for all}\; x,y
 \in
X;\\
 {\rm\bf cl2:}\quad&{}^{({}^xy)}x= {}^yx,
\quad\text{for all}\; x,y \in X; \quad\quad &{\rm\bf cr2:}\quad
&x^{(y^x)}= x^y, \quad\text{for all}\; x,y \in X.
\end{array}\]
$(X,r)$ is called \emph{cyclic} if it satisfies all cyclic
conditions.
\item Condition \textbf{lri} is defined as
 \[ \textbf{lri:}\quad
\quad ({}^xy)^x= y={}^x{(y^x)}, \;\text{for all} \quad x,y \in X.\]
In other words \textbf{lri} holds if and only if $(X,r)$ is
non-degenerate and $\Rcal_x=\Lcal_{x^{-1}}$ and $\Lcal_x =
\Rcal_{x^{-1}}$.
\end{enumerate}
\end{definition}

Symmetric groups and their braces with special conditions on the
actions like \textbf{lri} or \textbf{Raut} were studied first in
\cite{Tatyana}. Here we continue this study (i) for general symmetric groups $(G, r)$, and
(ii) under the additional assumption that the associated left brace
$(G,+,\cdot)$ is a two-sided brace.

\begin{definition} \cite{Tatyana} A left brace $(G, +, \cdot)$ satisfies condition \textbf{Raut} if
\[\textbf{Raut}: \quad (a+b)^{c}=a^{c}+b^{c}, \: \forall a, b, c \in G.\]
\end{definition}

Note that condition \textbf{lri} on  the symmetric group $(G, r)$
implies that the left and the right actions of the group $G$ upon itself are mutually inverse,
while condition \textbf{Raut}  links the two parallel structures-
the symmetric group structure and  the brace structure of $G$.

\begin{notation}
\label{notation} We shall use notation as in \cite{Tatyana}. As
usual, given a solution $(X,r)$,  $G = G(X,r)$ denotes its structure
group, and $\Gcal= \Gcal(X,r)$ denotes its permutation group. The
canonically associated symmetric groups will be denoted by $(G,
r_G)$ and $(\Gcal, r_{\Gcal})$, respectively. In the case when
$(X,r)$ is a multipermutation solution of level $m$ we shall write
$\mpl (X,r) = m.$ Given a two-sided brace $(G, +, \cdot)$,  the
associated Jacobson radical ring is denoted by $G_*= (G, +, *)$
\end{notation}

\section{Left braces $(G, +, \cdot)$, the operation $*$ and some identities}
We study symmetric groups $(G,r)$ and left braces $(G, +, \cdot)$
satisfying the identity  $(a*b)*c = a*(b*c)$,  for all $a,b,c \in
G$, or equivalently, $(G,*)$ is a semigroup with zero ($e=0$ is a zero element in $(G , *)$). Clearly, if
$(G, +, \cdot)$ is a two-sided brace, then $(G, *)$ is a semigroup.
In particular, we are
interested in the following questions.
\begin{questions}
\begin{enumerate}
\item
What can be said about symmetric sets $(X,r)$ for which some of the symmetric groups $G = G(X,r)$, or $\Gcal = \Gcal(X,r)$
has associative law for the operation $*$?
\item
Does it exist a left brace $(G, +, \cdot)$, such that $(G, *)$ is a
semigroup, but $(G, +, \cdot)$ is not a two-sided brace?
\end{enumerate}
 \end{questions}
It is known that if $(X,r)$ is a solution, then $G(X,r)$ is a
two-sided brace \emph{iff} $(X,r)$ is a trivial solution,
\cite[Theorem~6.3]{Tatyana}. We shall prove that in the special case
when $G= G(X,r)$ is the symmetric group of a solution $(X,r)$, $(G,
*)$ is a semigroup if and only if $G$ is a two-sided brace, and therefore $(X,r)$ is a trivial solution, see Corollary \ref{CorPr*}.

\begin{proposition}
\label{Pr*} Let $(G,r)$ be a symmetric group and  let $(G, +,
\cdot)$ be the corresponding left brace. Suppose $(G, *)$ is a
semigroup and the additive group $(G,+)$ has no elements of order
two.
 Then $(G, +, \cdot)$ is a two-sided brace, or equivalently,  $(G, +, *)$ is a Jacobson radical ring.
\end{proposition}
\begin{proof}
We shall prove that
\begin{equation}
\label{AgaId}
(-a)*b=-(a *b), \; \forall a,b \in G.
\end{equation}
By (\ref{rumpdef}), we have
\[
[a+(-a)+a*(-a)]*b = a*b+(-a)*b+ a*[(-a)*b], \quad \forall a,b \in G.
\]
This together with the obvious equality $[a+(-a)+a*(-a)]*b = [a*(-a)]*b$, and the associative law in $(G, *)$
imply
\[
a*[(-a)*b] = [a*(-a)]*b = a*b+(-a)*b+ a*[(-a)*b], \quad \forall a,b \in G.
\]
It follows that $a*b+(-a)*b=0$, so the identity (\ref{AgaId}) holds
in $G$. Note that $(G, +,*)$ satisfies the hypothesis of
\cite[Theorem~13]{Sm}, and therefore $(G,+,*)$ is a Jacobson radical
ring.
\end{proof}

An easy consequence of Proposition~\ref{Pr*} and
\cite[Corollari~5.16]{Tatyana} is the following result.
\begin{corollary}
\label{CorPr*}
Let $(X,r)$ be a solution, $(G,r_G)$, $(\Gcal, r_{\Gcal})$,  $(G,+,\cdot)$, $(\Gcal,+,\cdot)$ in usual notation.
\begin{enumerate}
\item  $(G, *)$ is a semigroup if and only if $(G, +, \cdot)$ is a two-sided brace, so in this case $(X,r)$ is the trivial solution.
\item Suppose  the additive group $(\Gcal,+)$ has no elements of order two  ($a + a \neq e, \forall a \in \Gcal$).
 Then $(\Gcal, *)$ is a semigroup if and only if it is a two-sided brace. Moreover, if $X$ is a finite set, then $(X,r)$,  is a multipermutation solution, and
 \[0 \leq \mpl (\Gcal, r_{\Gcal}) = m - 1\leq \mpl(X,r) \leq \mpl(G,r_G) = m < \infty.\]
\end{enumerate}
\end{corollary}

Recall that the series $G^n$ and $G^{(n)}$ of a left brace are defined by (\ref{Rideals1}), and (\ref{Rumpchain2}).

\begin{lemma}
\label{lem*}
Let $(G,+,\cdot)$ be a left brace.  Suppose that
$(G,*)$ is a semigroup. Then $G^n\subseteq G^{(n)}$ for all positive
integers $n$.
\end{lemma}
\begin{proof}
We shall use induction on $n$ to prove the equality of sets
\begin{equation}
\label{eq_series1}
G^n=\{\sum_{i=1}^kg_{i,1}*\cdots *g_{i,n}\mid k \mbox{ is a positive integer},\mbox{ and }g_{i,j}\in G \}.
\end{equation}
For $n=2$,  one has $G^2 = G*G=G^{(2)}$ by definition, thus \[G^2=\{\sum_{i=1}^kg_i*h_i\mid k \mbox{ is
a positive integer},\; g_{i},h_i\in G \}.\]
Let $n > 2$ and assume (\ref{eq_series1}) is true for all $m<n$.
By (\ref{Rumpchain2}) one has
\begin{equation}
\label{eq_series2}
G^n=G*G^{n-1}= \{\sum_{i=1}^kg_i*h_i\mid k \mbox{ is
a positive integer},\; g_{i}\in G,\; h_i\in G^{n-1} \}.
\end{equation}
By the induction hypothesis every pair $g \in G, h\in G^{n-1}$ satisfies
\begin{eqnarray*}
g*h&=&g*\sum_{i=1}^{k}g_{i,1}*\cdots *g_{i,n-1}\\
&=&\sum_{i=1}^kg*g_{i,1}*\cdots *g_{i,n-1},
\end{eqnarray*}
where $g_{i,j}\in G$.
This together with (\ref{eq_series2}) implies the desired equality of sets (\ref{eq_series1}).

It is clear that  $g_{1}*\cdots *g_{n}=(\dots (g_{1}*g_2)*\cdots)*g_n\in
G^{(n)}$, whenever $g_i \in G, 1 \leq i\leq n$. Therefore $G^n\subseteq G^{(n)}$.
\end{proof}

\begin{remark}
\label{rem1} Let $G$ be a set  with two operations "$\cdot$" and
"$+$" such that  $(G,\cdot)$ is a group, and $(G,+)$ is an abelian
group. (We do not assume $(G, +, \cdot)$ is a brace). Let $*$ be a
new operation on $G$ defined by (\ref{eqleftbrace*}).
\begin{enumerate}
\item $(G,+, \cdot)$ is a left brace if and only if   $(G, + , *)$ satisfies a left distributive law:
 $ a*(b+c) = a*b + a*c, \quad \forall\; a,b,c \in G$.
\item
 $(G,+, \cdot)$ is a right brace if and only if   $(G, + , *)$ satisfies a right distributive law:
$ (a+b)*c = a*c + b*c, \quad \forall\; a,b,c \in G$.
 \end{enumerate}
 \end{remark}

\begin{lemma}
\label{LemmaPro*}
Let $(G, +, \cdot)$ be a left brace, such that $(G, *)$  is  semigroup. If  $G^{n}= 0$ for some positive integer $n$ then
 $(G, +, \cdot)$ is a two-sided brace.
\end{lemma}
 \begin{proof}
 It follows from \cite[Lemma~15]{Agatasmok}  that for every
$a,b,c\in G$
 there are $d_{i}, d_{i}'\in G$ such that
\begin{equation}
\label{eqAgata}
(a+b)*c=a*c+b*c+\sum _{i=0}^{2n} (-1)^{i+1}((d_{i}*d_{i}')*c-d_{i}*(d_{i}'*c)).
\end{equation}
By hypothesis  $(G,*)$ is a semigroup, so $(d_{i}*d_{i}')*c-d_{i}*(d_{i}'*c)=0, \; 0 \leq i \leq 2n$, which together with
(\ref{eqAgata}) imply
$(a+b)*c=a*b+a*c$, for all $a,b,c \in G$.  Therefore, by Remark \ref{rem1} $(G, +, \cdot)$  is a two-sided brace.
 \end{proof}
\begin{proposition}
\label{Pro*} Let $(G,r)$ be a symmetric group of a finite
multipermutation level, $\mpl (G,r) = m$. Suppose $(G, *)$ is a
semigroup.  Then the following two conditions hold.
\begin{enumerate}
\item The left brace $(G, +, \cdot)$ is a two-sided brace, and hence $(G, +, *)$ is a Jacobson radical ring.
\item
The group $(G, \cdot)$ is nilpotent.
\end{enumerate}
\end{proposition}
\begin{proof}
By hypothesis $(G,r)$ has a finite multipermutation level, $\mpl
(G,r) = m$, so \cite[Proposition~6]{CGIS} implies that
$G^{(m+1)}=0$ and
$G^{(m)}\neq 0$.
 It follows from Lemma~\ref{lem*} that  $G^{m+1} \subseteq
G^{(m+1)}=0$. Clearly, the hypothesis of  Lemma  \ref{LemmaPro*} is satisfied, so $(G, +, \cdot)$ is a two-sided brace. This proves part (1) of the proposition.
The nilpotency of the the group $(G,\cdot)$ follows from \cite[Proposition~8]{Agatasmok}.
\end{proof}
\begin{question} Under the hypothesis of Proposition \ref{Pro*},  can we find an upper bound $B(m)$, depending on $m$, so that $G$ has nilpotency class $\leq B(m)$?
\end{question}
\begin{proposition}
\label{Pro*2} Let $(G,r)$ be a finite symmetric group such that $(G,
*)$ is a semigroup.
The following conditions are equivalent.
\begin{enumerate}
\item
$(G,r)$ has a finite multipermutation level, $\mpl (G,r) = m$.
\item $(G,\cdot)$ is a
nilpotent group.
\item $(G, +, \cdot)$ is a two-sided brace.
\end{enumerate}
\end{proposition}
\begin{proof}
The implications (1) $\Longrightarrow$  (2) and (1)
$\Longrightarrow$  (3) follow from Proposition~\ref{Pro*}. The
implication  (3) $\Longrightarrow$ (1) is well known, \cite{Ru07},
\cite{Tatyana}. (2) $\Longrightarrow$  (3). Assume the group $G$ is
nilpotent. Then \cite[Theorem~1]{Agatasmok} implies $G^n=0$ for some
positive integer
 $n$. It follows from Lemma  \ref{LemmaPro*} that $(G, +, \cdot)$ is a two-sided brace.
\end{proof}

Recall that $(X,r)$ is a multipermutation solution if and only if
the corresponding symmetric group  $(\Gcal, r_{\Gcal})$ has a finite
multipermutation level, \cite[Theorem~5.15]{Tatyana}.

\begin{corollary}
\label{CorPr**}
Let $(X,r)$ be a multipermutation solution, $\Gcal=\Gcal(X,r)$.
Then $(\Gcal, *)$ is a semigroup if and only if the left brace  $(\Gcal, +, \cdot)$ is a two-sided brace.
\end{corollary}

\begin{remark}
\label{Prop*}
Let $(G,r)$ be a symmetric group, and  let $(G, +, \cdot)$ be the corresponding left brace.
Suppose  that
$(G, +, * )$ is  a Jacobson radical ring generated by a finite set
$X= \{x_1, \cdots, x_n\}\subseteq G$.
If $(G, *)$ satisfies the identity
\[x*u*x = 0, \forall x \in X, \;  \; u \in G, (u = e \; \text{is possible}) \]
then the left brace $G$ is nilpotent of nilpotency class $\leq n+1.$
Moreover, $(G, r)$ has finite multipermutation level, $\mpl (G,r)
\leq n$.
\end{remark}
\begin{proof}
By assumption $(G, +, *)$ is a Jacobson radical ring.
 Therefore any element  from $ G^{(k)}, k \geq 1$,  can be written as a sum of elements $w$ of the form
$w = y_1*y_2*\cdots *y_s,  \; \; y_j \in X \bigcup \{e \}, \; 1\leq j \leq s, \; s \geq k.$ But $|X|=n$, hence   every such  element  $w \in G^{(n+1)}$
 has a subword  $x *a*x,$ where $x \in X,  a \in G$, or has the shape $w = u*e*v,\;  u, v \in G,$  so in each case $w= 0$ . Hence  $G^{(n+1)}= 0$,
and therefore, by \cite[Proposition~6]{CGIS}, $\mpl(G,r)\leq n$.
\end{proof}

\section{Symmetric sets $(X,r)$ whose associated groups and braces have special
properties}

It was proven in \cite[Theorem~8.2]{Tatyana},  that for a nontrivial
square-free solution $(X, r)$, with $G=G(X,r)$ one has $\mpl (X,r)
=\mpl (G,r_G) = 2$ if and only if $(G, r_G)$ satisfies condition
\textbf{lri}. We generalize  this result  for  arbitrary solutions
$(X, r)$.

\begin{theorem}
 \label{Thm_main}
 Let $(X,r)$ be a solution of arbitrary cardinality, $G=G(X,r)$, $(G, r_G)$, $(G, +, \cdot)$ in usual notation.
 The following conditions are equivalent.

 \begin{enumerate}
 \item
 \label{Thm_main1}
 $(G,r_G)$ is a non-trivial solution with condition \textbf{lri}.
 \item
  \label{Thm_main2}
 $(G,r_G)$ is a multipermutation solution of level $2$.
\item
    \label{Thm_main3}
  $G$ acts (nontrivially) upon itself  as automorphisms that is
\[
\Lcal_{(a^b)}= \Lcal_a, \quad  \forall\; a,b \in G,\; \text{and} \; \Lcal_a \neq id_G,\; \text{for some} \; a \in G.\]
  \item
   \label{Thm_main4}
   $(X,r)$ is a non-trivial solution with \textbf{lri} and the brace $(G, +, \cdot)$ satisfies
   \textbf{Raut}.
  \end{enumerate}
  Each of these conditions imply $\mpl (X,r) \leq 2.$
  \end{theorem}
  \begin{proof}
\cite[Proposition 7.13]{Tatyana} gives the implications
(\ref{Thm_main2}) $\Longleftrightarrow$ (\ref{Thm_main3})
$\Longrightarrow$ (\ref{Thm_main1}). The equivalence
(\ref{Thm_main1}) $\Longleftrightarrow$ (\ref{Thm_main4}) follows
from  \cite[Corollary 7.11]{Tatyana}.

(\ref{Thm_main1}) $ \Longrightarrow$ (\ref{Thm_main2}). Assume that
$(G, r_G)$ is a nontrivial solution which satisfies \textbf{lri}. We
shall show that $\Lcal_{({}^az)}= \Lcal_{z}$ for all $z \in X, a \in
G.$

By \cite[Proposition~2.25]{GIM08}, $G$ satisfies the cyclic
conditions. We use successively  {\bf ML0}, {\bf ML2}, {\bf lri} and {\bf cl2} to obtain
$$1={^{a}}(b^{-1}b)={^{a}}(b^{-1}){^{(a^{b^{-1}})}}b={^{a}}(b^{-1}){^{(^ba)}}b={^{a}}(b^{-1}){^{a}}b,$$
for all $a,b\in G$. Thus
\begin{equation}\label{invers}
^{a}(b^{-1})=({^{a}}b)^{-1}, \; \forall a,b\in G.
\end{equation}
 Let $x,y,z\in X$. Then condition {\bf lri} implies
\begin{eqnarray}\label{*}
(xy^{-1})^{z^{-1}}={^z}(xy^{-1}).
\end{eqnarray}
Note that $y^{-1}= -\left({}^{y^{-1}}y\right)=-\left(y^y\right)$. We now
compute each side of (\ref{*}).
For the left-hand side we obtain
\begin{eqnarray*}
(xy^{-1})^{z^{-1}}&=&x^{({^{(y^{-1})}}(z^{-1}))}(y^{-1})^{z^{-1}} \quad (\mbox{by \textbf{MR2} })\\
&=&{^{(z^{y})}}x(y^{-1})^{z^{-1}}\quad (\mbox{by {\bf lri} and (\ref{invers})})\\
&=&\left({^{(z^y)}}x\right)\left({}^{z}(-(y^{y}))\right)\\
&=& \left({}^{(z^y)}x\right)+{}^{\left({}^{(z^y)}x\right)}\left({}^{z}(-(y^{y}))\right) \quad (\mbox{by (\ref{eqoperation1})})\\
&=&
\left({}^{(z^y)}x\right)-{}^{\left({}^{(z^y)}x\right)}\left({}^{z}(y^{y})\right).
\end{eqnarray*}
Our computation of the right-hand side gives
\begin{eqnarray*}
{^z}(xy^{-1})&=&\left({}^{z}x\right)\cdot\left({}^{(z^x)}(y^{-1})\right) \quad (\mbox{by \textbf{ML2} })\\
&=&\left({}^{z}x\right)+{}^{\left({}^{z}x\right)}\left({}^{(z^x)}(-(y^{y}))\right)\quad (\mbox{by (\ref{eqoperation1})})\\
&=&\left({}^{z}x\right)-{}^{({}^{z}x)\cdot(z^x)}(y^{y})\\
&=&\left({}^{z}x\right)-{}^{(z\cdot x)}(y^{y}).
\end{eqnarray*}
Therefore the following equality holds in $G$
\begin{equation}\label{eq**}
\left({}^{(z^y)}x\right)-{}^{\left({}^{(z^y)}x\right)}\left({}^{z}(y^{y})\right) = \left({}^{z}x\right)-{}^{(z\cdot x)}(y^{y}).
\end{equation}
Note that $(G,+)$ is
a free abelian group with a basis $X$, and $\left({}^{(z^y)}x\right)$,
${}^{\left({}^{(z^y)}x\right)}\left({}^{z}(y^{y})\right)$,
$\left({}^{z}x\right)$, ${}^{(z\cdot x)}(y^{y})\in X$.  Hence the equality (\ref{eq**}) implies that either
\begin{equation}
\label{eqlri1}
{}^{(z^y)}x={}^{z}x,
\end{equation}
or
\begin{equation}
\label{eqlri1a}
{}^{(z^y)}x    \neq  {}^{z}x,
\quad\mbox{and}\quad \left({}^{z}x\right)-{}^{(z\cdot x)}(y^{y})= 0.
\end{equation}
We claim that (\ref{eqlri1a}) is impossible. Indeed, $0 = 1$ in $G$,
hence  ${^z}(xy^{-1}) = \left({}^{z}x\right)-{}^{(z\cdot x)}(y^{y})=
0$ implies ${^z}(xy^{-1})= 1$, which by \textbf{ML0} gives $xy^{-1}=
1$, and therefore $x=y$. Now the cyclic condition implies
$\left({}^{(z^y)}x\right)=\left({}^{(z^x)}x\right) =
\left({}^{z}x\right)$, which contradicts  (\ref{eqlri1a}). It
follows then that ${}^{({}^{y}z)}x= {}^{z}x$, for all $x,y,z \in X.$
This, together with \textbf{lri} and (\ref{invers}), imply
"enforced" cyclic conditions
\begin{equation}
\label{eqlri5}
  \begin{array}{lll}
  &{}^{(z^y)}x= {}^{z}x, \quad & {}^{({}^{y}z)}x= {}^{z}x \\
  &  x^{({}^{y}z)} = x^z,   \quad  & x^{(z^y)} = x^z,
  \end{array}
\end{equation}
for all $x, y, z \in X^{\star}=X\cup X^{-1}$, where  $X^{-1}=\{
x^{-1}\mid x \in X \}.$

We use induction on the length $|a|$ of $a \in G$ to show that
\begin{equation}
\label{eqlri6} {}^{({}^{y}z)}a= {}^{z}a, \quad   {}^{(z^y)}a=
{}^{z}a, \quad \forall a \in G,\;\forall y,z \in X^{\star}.
\end{equation}
The base for induction follows from (\ref{eqlri5}). Assume (\ref{eqlri6}) is in force for all $a\in G$ with $1 \leq |a|\leq k.$
Suppose $a \in G, 2 \leq |a|= k+1$, then $a = tb, t \in X^{\star}, b \in G, |b|= k.$
We use \textbf{ML2} and the inductive hypothesis (IH) to yield:
\[
  \begin{array}{llll}
  {}^{({}^{y}z)}a &=&  {}^{({}^{y}z)}{(tb)} = ({}^{({}^{y}z)}{(t)})({}^{(({}^{y}z)^t)}{(b)})&\\
                  &=&  ({}^{({}^{y}z)}{(t)})({}^{(({}^{y}z))}{(b)}) & \text{by IH}\\
                  &=&  ({}^zt)({}^zb) & \text{by IH} \\
   {}^za          &=&  {}^z{(tb)} = ({}^z{(t)})({}^{(z^t)}{(b)})&\\
                  &=&  ({}^zt)({}^zb) & \text{by IH}.
  \end{array}
\]
This implies the first equality in (\ref{eqlri6}) for all $a\in G, y, z \in X^{\star}.$ Using \textbf{lri} one deduces that the second equality in (\ref{eqlri6}) is also in force.
Similar technique "extends" (\ref{eqlri6}) on the whole group $G$, so that the following equalities hold:
\[
{}^{({}^{b}c)}a= {}^{c}a, \quad   {}^{(b^c)}a= {}^ca \quad \forall a, b, c \in G.
\]

It follows from   \cite[Lemma~7.12]{Tatyana} that the symmetric
group $(G, r_G)$ satisfies the four equivalent conditions.
\begin{equation}
  \label{MLauteq2}
  \begin{array}{lllll}
  &\text{\textbf{(i)}} \quad &\Lcal_{({}^ba)}=\Lcal_a, \; \forall\; a,b \in G;\quad
  &\text{\textbf{(ii)}} \quad     &\Lcal_{(a^b)}=\Lcal_a,   \; \forall\; a,b\in G;\\
 &\text{\textbf{(iii)}} \quad     &\Rcal_{({}^ba)}=\Rcal_a,\;  \forall\; a,b  \in G;\quad
 &\text{\textbf{(iv)}}       &\Rcal_{(a^b)}=\Rcal_a,\;     \forall\; a,b \in G.
  \end{array}
  \end{equation}
By  \cite[Proposition~7.13]{Tatyana} each of the conditions (i)
through (iv) is equivalent to (\ref{Thm_main2}). We have shown the
implication (\ref{Thm_main1}) $\Longrightarrow$ (\ref{Thm_main2}),
so $\mpl (G,r_G) = 2.$ By \cite[ Theorem~5.15]{Tatyana}, one has
$\mpl (G,r_G) -1 \leq \mpl (X,r) \leq \mpl (G,r_G)$, and therefore  $\mpl
(X,r) \leq 2.$
\end{proof}

Suppose $(X,r)$ is a solution with \textbf{lri}, and  $(G,r_G)$ is its associated symmetric group.
Let $(\overline{G}, r_{\overline{G}})$ be a symmetric group, and assume there is a braiding-preserving
map (homomorphism of solutions)
\[\mu: X \longrightarrow \overline{G} \quad x \mapsto \overline{x} \in \overline{G}\]
Then by \cite[Theorem~9]{LYZ}, the map $\mu$ extends canonically to
a braiding preserving group homomorphism (that is a homomorphism of
symmetric groups)
\[\mu: (G,r_G) \longrightarrow (\overline{G}, r_{\overline{G}})  \quad a \mapsto \overline{a} \in \overline{G}.\]
Moreover, if $\overline{X}= \mu (X)$ is a set of (multiplicative) generators of $\overline{G}$ then $ \mu: G \longrightarrow
\overline{G}$
is an epimorphism of symmetric groups.

The following result is a generalization of
\cite[Theorem~7.10(2)]{Tatyana}.

 \begin{theorem}
 \label{Thm_braces2_lemma}
 Let $(X,r)$ be a symmetric set with \textbf{lri} (not necessarily finite),
 let $(\overline{G}, \overline{r})$ be a symmetric group, and let $(\overline{G}, +, .)$ be the associated left brace.
 Assume there is a braiding-preserving
map (homomorphism of solutions)
\[\mu: X \longrightarrow \overline{G}, \quad x \mapsto \overline{x} \in \overline{G},\]
such hat the image   $\mu (X) = \overline{X}$, is an $\overline{r}$-invariant subset of $(\overline{G}, \overline{r})$ and
generates the (multiplicative) group $\overline{G}$.
The following conditions are equivalent on $\overline{G}$.
 \begin{enumerate}
 \item  The left brace $(\overline{G}, +, \cdot)$  satisfies condition
     \textbf{Raut}.
 \item $(\overline{G},r_{\overline{G}})$ satisfies condition \textbf{lri}.
\end{enumerate}
\end{theorem}
\begin{proof}
(1) $\Longrightarrow $ (2). Suppose  $(X,r)$ satisfies \textbf{lri} and $\overline{G}$
satisfies  condition \textbf{Raut}.

Recall that $X^{\star}= \{x\mid x \in X \;\text{or}\; x^{-1} \in X
\}$. By  \cite[Proposition~7.6]{Tatyana}, condition \textbf{lri} on
$(X,r)$ extends to
 \label{VIPLemma}\begin{equation}
  \label{lri*eq}
 {\rm\bf lri\star:} \quad  {}^a{(x^a)}= x =({}^ax)^a,\quad \forall\;x \in
 X^{\star},  \; a \in G.
 \end{equation}
Denote by $\overline{X^{\star}}= \mu (X^{\star})$ the image of $X^{\star}$ in $\overline{G}$. (It is possible that $\overline{X^{\star}}$
contains the unit
$1= 1_{\overline{G}}$ of the group $\overline{G}$).

We shall extend \textbf{lri} on the symmetric group  $(\overline{G},r_{\overline{G}})$ in two steps.
\textbf{1.}  We show that
\begin{equation}
\label{lri7}
{}^{(\overline{a})^{-1}}\overline{u} = \overline{u}^{\overline{a}} \quad \text{for all} \; \overline{a} \in \overline{X^{\star}}, \;
\overline{u} \in \overline{G}.
\end{equation}

For $\overline{u} \in \overline{G}$ we consider $u \in G$ of minimal length, such that $\mu(u) = \overline{u}$.
 Without loss of generality we may assume that $\overline{u} \neq 1$ (this follows from \textbf{ML0} and \textbf{MR0}).
 We use induction on the minimal length $|u|$ of  $u$, with $\mu(u) = \overline{u}$.
Condition ${\rm{\bf lri\star}}$, (\ref{lri*eq}) gives the base for
induction. Assume (\ref{lri7}) holds for all $\overline{a} \in
\overline{X^{\star}}$ and all $\overline{u} \in \overline{G}$, where
$\overline{u} = \mu(u), |u|\leq n$. Let $\overline{a} \in
\overline{X^{\star}}$  and suppose $\overline{w} \in \overline{G}$,
where $\overline{w} = \mu(w), |w| = n+1$. A reduced form of $w$ can
be written as $w = xu$, where $x \in X^{\star}$,  $u \in G, \;|u| =
n $. We present $\overline{w}^{\overline{a}}$ as
\[
 \overline{w}^{\overline{a}}=(\overline{xu})^{\overline{a}}=((\overline{x})(\overline{u}))^{\overline{a}}
 =(\overline{x}+{}^{\overline{x}}{(\overline{u})})^{\overline{a}},
\]
and consider the following equalities in $G$:
\begin{equation}
\label{lri9}
\begin{array}{llll}
 \overline{w}^{\overline{a}}&=& (\overline{x}+{}^{\overline{x}}{(\overline{u})})^{\overline{a}}&\\
   &=& (\overline{x})^{\overline{a}}+({}^{\overline{x}}{(\overline{u})})^{\overline{a}}
   \quad & \text{by \textbf{Raut}} \\
   &=& {}^{(\overline{a})^{-1}}{(\overline{x})} + {}^{({\overline{a})}^{-1}}{({}^{\overline{x}}{(\overline{u})})}&\text{by IH}\\
   &=& {}^{(\overline{a})^{-1}}{(\overline{x} + {}^{\overline{x}}{(\overline{u})})}&\text{by \textbf{Laut}}\\
   &=& {}^{{(\overline{a})}^{-1}}{(\overline{x}\cdot\overline{u})} &\\
   &=& {}^{{(\overline{a})}^{-1}}{(\overline{w})}, &
 \end{array}
\end{equation}
where IH is the inductive assumption.
This verifies (\ref{lri7})
for all $\overline{a} \in \overline{X^{\star}}$,  and all $\overline{u} \in \overline{G}$.
Clearly, (\ref{lri7})  is equivalent to
\begin{equation}
\label{lri10} {}^{\overline{a}} {(\overline{u}^{\overline{a}})}=
\overline{u} \quad \forall \; \overline{a}, \overline{u}, \;
\text{where}\;  \overline{a} \in \overline{X^{\star}},\;
\overline{u} \in \overline{G}.
\end{equation}
\textbf{2.} We  shall extend (\ref{lri10}) for all $\overline{a}\in
\overline{G}.$ We  use induction again, this time on the minimal
length of the elements $a \in G$ with  $\mu(a) = \overline{a}$. The
base of the induction is given by (\ref{lri10}). Assume
${}^{\overline{a}}{(\overline{u}^{\overline{a}})}= \overline{u}$ for
all $\overline{a}, \overline{u} \in \overline{G},$ where there is an
$a \in G$, such that $\mu(a) = \overline{a}$, and  $|a|\leq n.$ Let
$\overline{a}, \overline{u} \in \overline{G},$ and assume the
minimal length of the $a$'s with $\mu(a)= \overline{a}$ is $|a|=
n+1$.  Then $a = bx, \; x \in X^{\star}, b \in G, |b|= n$. The
following equalities hold:
\begin{equation}
\label{lri12}
\begin{array}{llll}
{}^{\overline{a}} {({\overline{u}}^{\overline{a}})} &=& {}^{\overline{bx}} {({\overline{u}}^{\overline{bx}})} \quad &\\
&=& {}^{\overline{b}}
{({}^{\overline{x}}{({(\overline{u}^{\overline{b}})}^{\overline{x}})})}
            \quad &\\
&=& {}^{\overline{b}}{(\overline{u}^{\overline{b}})}  & \text{by IH} \\
&=& \overline{u} &\text{by IH}.
 \end{array}
\end{equation}
 This verifies
 \[
{}^{\overline{a}}{(\overline{u}^{\overline{a}})}
= \overline{u},\quad \forall \; \overline{a}, \; \overline{u}, \in \overline{G}.
\]
The remaining identity \[({}^{\overline{a}}{\overline{u}})^{\overline{a}}= \overline{u},\quad \forall \; \overline{a},\; \overline{u}, \in
\overline{G}\]
is straightforward.
We have shown that the symmetric group $(\overline{G}, \overline{r})$
satisfies condition \textbf{lri}.

(2) $\Longrightarrow$ (1). It follows from \cite[Theorem~7.10]{Tatyana}
 that  condition \textbf{lri} on an arbitrary symmetric group implies \textbf{Raut} on the
corresponding left brace.
\end{proof}

\begin{corollary}
 \label{Thm_braces2}
 Let $(X,r)$ be a symmetric set with \textbf{lri} (not necessarily finite), notation as usual.
The symmetric group $(\Gcal, r_{\Gcal})$ satisfies condition
\textbf{lri} if and only if the associated left brace
 $(\Gcal, +,\cdot )$ satisfies condition
     \textbf{Raut}.
\end{corollary}
\begin{proof}
 The map
\[\Lcal: (G, r_G) \longrightarrow  (\Gcal, r_{\Gcal}), \quad x\mapsto \Lcal_{x}, \]
 is a braiding preserving homomorphisms of symmetric groups, the image
 $\Lcal(X)$
 generates the permutation group
 $\Gcal$.
 So the hypothesis of Theorem~\ref{Thm_braces2_lemma}
 is satisfied for
 $\mu = \Lcal$, which implies the equivalence of \textbf{lri} and \textbf{Raut}  on $(\Gcal, r_{\Gcal})$.
\end{proof}

The next corollary follows from  Corollary~\ref{Thm_braces2}, and
\cite[Theorems~8.5 and~5.15]{Tatyana}.
\begin{corollary}
Suppose $(X,r)$ is a finite square-free solution, notation as usual.
If  the symmetric group $(\Gcal, r_{\Gcal})$ satisfies condition
 \textbf{Raut}, then $(X,r)$ is a multipermutation solution of level $m < |X|$, and
\[ \mpl (X,r) = \mpl (G, r_G) = m, \; \; \mpl(\Gcal, r_{\Gcal}) = m-1.  \]
\end{corollary}

\section{Conditions \textbf{lri} and \textbf{Raut} on symmetric groups with two-sided braces}

In this section we study symmetric groups $(G,r)$ whose associated braces $(G, +, .)$ are two-sided, or equivalently
$G_* = (G, +, *)$ are Jacobson radical rings. We present each of the conditions \textbf{lri} and \textbf{Raut} in terms of identities on the radical ring $G_*$.

We start with some useful results interpreting various conditions on a symmetric group $(G,r)$ in terms of the operation $*$

\begin{lemma}
\label{identities_Lemma*}
Let $(G,r)$ be a symmetric group. The the following conditions hold.
\begin{enumerate}
\item
$G$ satisfies the identity
\begin{equation}
\label{identity_eq01}
(a* c + c)* a^c  + a^c = a, \quad  \forall a, c \in G.
\end{equation}
\item Suppose the associated brace $(G, +, \cdot)$ is two-sided. Then the Jacobson radical ring $G_*= (G, +, *)$ satisfies the identities
 \begin{equation}
\label{identity_eq1}
a* c* a^c + c* a^c + a^c = a, \quad  \forall a, c \in G.
\end{equation}
\begin{equation}
\label{eqlri4}
{}^{({}^ac)}a =  a* c* a + {}^ca, \quad \forall a, c \in G.
\end{equation}
\end{enumerate}
\end{lemma}
\begin{proof}
(1) The map $r$ is involutive, which is equivalent to the following conditions on the actions
\begin{equation}
\label{involeq}
{}^{{}^ac}{(a^c)}= a, \quad ({}^ac)^{a^c} = c, \; \;\forall \; a,c \in G.
\end{equation}
We use  (\ref{eqoperation2}) to present the first equality in terms of the the operations $+, *$
and  yield
\[\begin{array}{llll}
a &=& {}^{{}^ac}{(a^c)}&\\
  &=& ({}^ac)* (a^c) + a^c &\\
  &=& (a* c +c )*(a^c) + a^c. &\\
   \end{array},
\]
so (\ref{identity_eq01}) holds.

(2) Suppose the associated brace is two-sided, let $G_*= (G, +, *)$. Clearly, the identities (\ref{identity_eq01}) and (\ref{identity_eq1}) are equivalent.
Let $a,c \in G$ then
\[{}^{({}^ac)}a =    ({}^ac)* a + a = (a* c+ c)* a+a = a* c* a + c* a + a  =a* c* a +
{}^ca, \]
which proves (\ref{eqlri4})
\end{proof}

\begin{proposition}
\label{cl1_Prop*}
Suppose $(G,r)$ is a symmetric group.
The following conditions are equivalent
\begin{enumerate}
\item    $G$ satisfies the identity
\begin{equation}
\label{id_cl1*}
(c^a)*a= c*a,  \; \;\forall \; a,c \in G.
\end{equation}
\item
 $(G,r)$ satisfy the cyclic condition \textbf{cl1}:
\begin{equation}
\label{id_cl1}
\text{\textbf{cl1}}:\quad {}^{c^a}a = {}^ca  \; \;\forall \; a,c \in G.
\end{equation}
\item
$(G,r)$ satisfies \textbf{lri}.
\item
$G$ satisfies all cyclic conditions,  see Definition \ref{lri&cl}.
\end{enumerate}
\end{proposition}
\begin{proof}
The equivalence (1) $\Longleftrightarrow$ (2) follows straightforwardly from the equalities
\begin{equation}
\label{cl1eq*}
{}^{c^a}a= (c^a)*a + a, \quad {}^ca = c*a + a,  \; a, c \in G
\end{equation}
(2)  $\Longrightarrow$ (3). Assume  \textbf{cl1} is in force. We shall verify the first and the second \textbf{lri}
equalities
\[ \textbf{lri1}:\quad
({}^ca)^c= a, \; \forall a, c \in G, \quad  \quad \textbf{lri2}:\quad {}^c{(a^c)}= a \; \;   \forall a, c \in G.\]

 Let $a, c \in G$.
By the non-degeneracy there exists  $b \in G$, with $c =b^a$. We use
(\ref{involeq}) and \textbf{cl1} to obtain $a = ({}^ba)^{b^a} =
({}^{b^a}a)^{b^a} = ({}^ca)^c$. This proves \textbf{lri1}. It
follows from the non-degeneracy again that there exists $d \in G$,
such that $a = {}^cd$. One has $ {}^c{(a^c)} = {}^c{(({}^cd)^c)}=
{}^cd = a, $ so the equality \textbf{lri2} is also in force.

(3)  $\Longrightarrow$  (2). Let $a,c \in G$. Then \textbf{lri1} and
(\ref{involeq})  imply $({}^{c^a}a)^{c^a}= a = ({}^ca)^{c^a}$. By
the non-degeneracy ${}^{c^a}a = {}^ca$, which proves \textbf{cl1}.
We have shown the equivalence of conditions (1), (2), and  (3). The
equivalence of (3) and (4) follows from \cite[Lemma~2.24]{GIM08}.
\end{proof}

\begin{theorem}\label{ThmLRI}
Let $(G,r)$ be a symmetric group with a two-sided  associated brace $(G, +,.)$, and let $G_*= (G, +, *)$ be the corresponding
 Jacobson radical ring. The following conditions are equivalent.
\begin{enumerate}
\item
\label{ThmLRI1}
$G_*$ satisfies the identity
\begin{equation}
\label{ThmLRI4}
a* c * a = 0, \quad \forall a, c \in G.
\end{equation}
\item
\label{ThmLRI1a}
$G_*$ satisfies the identity (\ref{id_cl1*}).
\item
\label{ThmLRI2}
The symmetric group $(G,r)$ satisfies conditions \textbf{lri}.
\item
\label{ThmLRI3} The symmetric group $(G,r)$ satisfies all cyclic
conditions.
\end{enumerate}
 \end{theorem}
\begin{proof}
The equivalence of conditions (\ref{ThmLRI1a}), (\ref{ThmLRI2}), and
(\ref{ThmLRI3}) follows from Proposition~\ref{cl1_Prop*}. By Lemma
\ref{identities_Lemma*} $G$ satisfies the identity (\ref{eqlri4})
which implies the equivalence
\[
[ {}^{({}^ac)}a = {}^ca, \; \forall \; a, c \in G] \Longleftrightarrow [ a* c* a = 0,  \;  \forall \; a, c \in G].
\]
Now the implication
(\ref{ThmLRI3}) $\; \Longrightarrow $ (\ref{ThmLRI1}) is straightforward.
We shall prove (\ref{ThmLRI1}) $\; \Longrightarrow $ (\ref{ThmLRI1a}). Assume (\ref{ThmLRI4}) holds.
By Lemma \ref{identities_Lemma*} $G$ satisfies the identity
\[
a = a* c* a^c + c* a^c + a^c, \quad  \forall \; a, c \in G.
\]
Hence
\[
\begin{array}{lll}
a*c &=& (a* c* a^c + c* a^c + a^c)*c \\
    &=& a* (c* a^c *c) + c* a^c*c + a^c*c\\
    &=& a^c*c \quad \quad \quad\quad \quad \quad \text{by (\ref{ThmLRI4})},
\end{array}
\]
which proves (\ref{ThmLRI1a}).
We have verified the equivalence of conditions (1), (2), (3) and (4).
\end{proof}

\begin{corollary}
 Suppose  $(G, r)$ is a symmetric group of arbitrary cardinality, such that
\begin{itemize}
\item[(i)]
$(G, +, \cdot )$ is a
two-sided brace, so $G_*= (G, +, *)$ is the corresponding Jacobson radical ring;
\item[(ii)]   $G_*$  is finitely generated (as a ring)  by a set $X$
of $N$ generators (equivalently, the group $(G,\cdot)$ is finitely
generated); \item[(iii)]  $(G,r)$ satisfies \textbf{lri}.
\end{itemize}
Then  the following conditions hold.
\begin{enumerate}
\item  $a*G*a = 0$,  for every $a \in G.$ \item The ring $G_*$ is
nilpotent with level of nilpotency $\leq N+1$. \item $ (G, r)$ has
multipermutation level  $\mpl (G,r) \leq N$.
\end{enumerate}
\end{corollary}
\begin{proof}
By Theorem \ref{ThmLRI} condition \textbf{lri} on $(G,r)$ implies
the identity $a*b*a =0$, for all $a,b \in G,$ so (1) is in force.  Conditions
(2) and (3) follow straightforwardly from Remark~\ref{Prop*}.
\end{proof}

\begin{theorem}\label{ThmE}
Let $G= (G, r)$ be a symmetric group. Assume its associated left brace $(G, +, .)$ is a two-sided brace,
and let $G_* =(G, +, *)$ be the corresponding Jacobson radical ring.
\begin{enumerate}
\item
Let $a,b,c \in G,$  $u=u(a,b,c) =(a+b)c, \; w=w(a,b,c)=  ({}^{(a+b)}c)(a^{c}+b^{c})$.
Then there is an equality
\begin{equation}
\label{w}
 w = a* c* b^{c}+  b* c * a^{c}+ u.
\end{equation}
\item
$G$ satisfies condition \textbf{Raut} if and only if the following identity is in force
\begin{equation}
\label{RautJacobs1}
a* c * b^c + b* c * a^c=0,  \quad \forall a, b, c \in G.
\end{equation}
\end{enumerate}
\end{theorem}
\begin{proof}

(1). We compute $u$ and $w$ as elements of the radical ring $G_*$.
One has $u=(a+b)c=(a+b)* c   +  a + b + c $, hence
\begin{equation}
\label{Rauteq3}
  u=a * c + b * c + a + b + c.
\end{equation}
Now we compute $w$:
\[\begin{array}{llll}
w&=&({}^{(a+b)}c)(a^{c}+b^{c})&\\
 &=&({}^{(a+b)}c)*(a^{c}+b^{c})+ {}^{(a+b)}c + a^{c}+b^{c}&\\
  &=&((a+b)* c + c)*(a^{c}+b^{c})+ ((a+b)* c + c) + a^{c}+b^{c}&\\
  &=&(a* c+ b* c + c)*(a^{c}+b^{c})+ ((a+b)* c + c) + a^{c}+b^{c}&\\
  &=&a* c* a^{c} +a* c* b^{c}+  b* c * a^{c}+ b* c * b^{c}+c * a^{c}+  c *
  b^{c}\\
  && + a* c +b* c + c + a^{c}+b^{c}&\\
  &=&[a* c* a^{c}+c * a^{c}+ a^{c}] + [b* c * b^{c}+ c *
  b^{c}+b^{c}]\\
  &&+a* c* b^{c}+  b* c * a^{c}
   + a* c +b* c + c &\\
 &=&a + b +a* c* b^{c}+  b* c * a^{c}+ a* c +b* c + c \quad\text{(we have applied (\ref{identity_eq1}) twice)}&\\
 &=&(a* c* b^{c})+  (b* c * a^{c})+ (a* c +b* c +a +b + c) &\\
 &=&(a* c* b^{c})+  (b* c * a^{c})+ u \quad \quad \quad \quad \text{(by (\ref{Rauteq3}))},&\\
   \end{array}
\]
which proves (1).

(2).  Note that condition \textbf{Raut} holds in $G$  \emph{iff}
\begin{equation}
\label{RautJacobs1a}
 (a+b)c =({}^{(a+b)}c)((a+b)^{c})=({}^{(a+b)}c)(a^{c}+b^{c}), \;  \forall a, b, c \in G.
 \end{equation}
In other words (in notation as above) condition \textbf{Raut} in $G$ is equivalent to
\[u(a,b,c)= w(a,b,c), \;\forall a,b,c \in G.\] This together with (\ref{w}) implies that
$G$ satisfies \textbf{Raut} if and only if the identity (\ref{RautJacobs1}) is in force.
\end{proof}

\section{Graded Jacobson radical rings $(G, +, *)$, their braces and symmetric groups}
In this section we consider graded Jacobson radical rings $R =
(R,+,*)$.

\begin{convention}
 \label{convRadring-twosidedbrace}
 To each Jacobson radical ring $R = (R, +, *)$, by convention we associate canonically a symmetric group $(R, r)$ and a  two-sided brace $(R, +, \cdot)$
 with operations and actions satisfying
\begin{equation}
\label{eqgeneral}
\begin{array}{l}
a\cdot b = a* b +a +b, \\
{ }^ab = a*b + b  = a\cdot b -a,\quad a^b = {}^{({}^ab)^{-1}}a,\\
\quad \quad  \forall \; a,b \in R.
\end{array}
\end{equation}
 Conversely, if $(G,r)$ is a symmetric group whose left brace $(G, +, \cdot)$ is a two-sided brace,
by convention we associate to $G$ the corresponding Jacobson radical ring $G_* = (G, +, *)$.
\end{convention}

By a graded ring we shall mean a ring graded by the additive
semigroup of positive integers.
 Thus  a graded
Jacobson radical ring  $R=(R, +, *)$ is presented as
\[ R=\oplus_{i=1}^{\infty} R_{i}, \; \text{where}\; R_i*R_j \subseteq R_{i+j}, \;  0 \in R_j, \; i, j\geq 1. \]
As usual, each element $a\in R_j, a \neq 0,$ is called \emph{a
homogeneous element of} \emph{degree $j$}, by convention the zero
element $0$ has degree $0$.

For consistency with our notation
 the operation multiplication in $R$ is denoted by $*$ (the ring $R$ does not have unit element with respect to the operation $*$).

\begin{proposition}
\label{PropGraded_mpl}
Let $(G,r)$ be a symmetric group, such that the associated left brace $(G, +, \cdot)$ is two-sided.
Suppose the associated Jacobson radical ring $G_*= (G, +, *)$ is graded: $G_*=\oplus_{i=1}^{\infty} G_{i}$, and is generated as a ring by the
set $V\subseteq G_{1}$.
Then $\mpl (G, r) = m$ if and only if $G_m \neq 0$ and $G_{i}=0, \; \forall \;  i \geq m+1$.
\end{proposition}
\begin{proof}
Consider the chain of ideals  $G^{(1)} = G, \; G^{(n+1)} =G^{(n)}*G,
n \geq 1 $, see (\ref{Rideals1}). One has  $G_i \subseteq G^{(k)}$,
for all $i \geq k$,  moreover
\begin{equation}
\label{PropGraded_eq1}
G^{(k)} = \oplus_{i\geq k} G_{i}, \; \forall \; k \geq 1.
\end{equation}
By \cite[Proposition~6]{CGIS}, the symmetric group $(G,r)$ has
finite multipermutation level $\mpl(G,r)=m<\infty $ if and only if
$G^{(m+1)}=0$ and $G^{(m)}\neq 0$. This together with
(\ref{PropGraded_eq1}) imply that $\mpl(G,r)=m$ if and only if $G_m
\neq 0,$ and $G_i = 0,$ for all $i \geq m+1.$
\end{proof}

\begin{remark}
\label{remark_E} Let $R$ be a graded Jacobson radical ring. Suppose
$a, b, c \in R$ are nonzero elements, and $a$ is a homogeneous
element of degree $i$, that is $a\in R_i$. Then it is clear that
\begin{equation}
\label{RautJacobs_graded06}
\begin{array}{llll}
(i) \quad {}^{b}{a} = a+ \tilde{a}, \quad \text{where}\; \;\widetilde{a}= b*a \in \oplus_{j>i} R_{j};\\
(ii) \quad a^c = a + \tilde{a}, \quad \text{where}\; \;\tilde{a} = (({}^{a}{c})^{-1})*a \in \oplus_{j>i} R_{j}
\end{array}
\end{equation}
 \end{remark}

\begin{lemma}
\label{lemmaPropE} Let $R_*= (R, +, *)$ be a graded Jacobson
radical ring, $R_*=\oplus_{i=1}^{\infty} R_{i}$. Let $(R, +,\cdot)$
be the associated two-sided brace, and let $(R, r)$ be the
associated symmetric group.  Suppose the brace $R$ satisfies
condition \textbf{Raut}.
\begin{enumerate}
\item
The following equality holds for homogeneous elements of $R$:
\begin{equation}
\label{RautJacobs_graded1a}
a_i* c_j* b_k+ b_k* c_j *a_i=0, \quad\forall \; a_i \in R_i,\; c_j \in R_j,\; b_k  \in R_k, \; i,j,k \geq 1.
\end{equation}
\item Moreover, if the additive group $(R,+)$ has no elements of order two, then
\begin{equation}
\label{RautJacobs_graded2a}
a_i* c_j* a_i=0, \quad \forall \; a_i \in R_i, \; c_j \in R_j,  i,j \geq 1.
\end{equation}
\end{enumerate}
\end{lemma}

\begin{proof}
(1)  Let  $a \in R_i  , c \in R_j, b  \in R_k, \;  i,j,k \geq 1$ be
non-zero elements (we omit the indices of $a,b,c$ for simplicity of
notation). Consider he equalities
\[\begin{array}{llll}
0 &=&a* c * (b^c) + b* c * (a^c) &\text{by Theorem \ref{ThmE}}\\
  &=&a* c * [b + \tilde{b}]  + b* c * [a + \tilde{a}] &\text{see Remark \ref{remark_E}}\\
  &=&[a* c * b +b* c * a] + [a* c * \tilde{b} +b* c * \tilde{a}] = f + g, &\\
  \end{array}
\]
where $f= a* c * b +b* c * a$, and $g =a* c * \tilde{b} +b* c * \tilde{a}$.
 Clearly, $f \in R_{i+j+k}$, and $g\in \oplus_{m\geq i+j+k+1} R_{m}$, see  Remark \ref{remark_E}.
 But $R$ is a graded ring, hence the equality $f + g=0$ holds \emph{iff} $f = 0$ and $g = 0.$
This proves (\ref{RautJacobs_graded1a}).

(2) Assume now that $(R,+)$ has no elements of order two, and let
 $a= a_i \in R_i , c= c_j \in R_j, i, j \geq 1$.
 Then we set $k=i$, $b_k = a$ in (\ref{RautJacobs_graded1a}) and obtain
 \[a_i*c_j*a_i + a_i*c_j*a_i= 0,\]
which  implies the desired equality $a_i*c_j*a_i= 0$, $\forall i, j
\geq 1.$
\end{proof}

\begin{theorem}
\label{PropE} Let $R_*=(R, +, *)$ be a graded Jacobson radical ring,
$R_*=\oplus_{i=1}^{\infty} R_{i}$. Let $(R, +,\cdot)$, and  $(R,
r)$, respectively,  be the associated two-sided brace and the
corresponding symmetric group. Suppose the additive group $(R,+)$
has no elements of order two. The following two conditions are
equivalent.
\begin{enumerate}
\item
The brace $(R, +,\cdot)$ satisfies condition \textbf{Raut}.
\item
The symmetric group $(R, r)$ satisfies condition \textbf{lri}.
\end{enumerate}
\end{theorem}

\begin{proof}
(1) $\Longrightarrow $ (2). Assume the brace $(R, +,\cdot)$ satisfies \textbf{Raut}.
We shall prove that $a*c*a =0, \forall a, c \in R$.
Suppose $a,c \in R$ and present each of them as a finite sums of homogeneous components.  So $a= \sum_{i=1}^{N}a_i$,
$c= \sum_{j=1}^{N}c_j$,
 where $a_i,  c_i \in R_i, i \geq 1,$ and there are natural numbers $N_a, N_c,$ such that $a_i = 0$ for all $i \geq N_a,$
 $c_j = 0$ for all $j \geq N_c.$ Set $N= \max (N_a, N_c)$.

Lemma \ref{lemmaPropE} implies the following equalities
\begin{equation}
\label{RautJacobs_graded3a}
 a_i*c_j*a_k +  a_k*c_j*a_i = 0\quad \mbox{ and }\quad
a_i*c_j*a_i= 0,
\end{equation}
for all  $i,j,k$ with $1\leq i,j, k \leq N$. Then, by
(\ref{RautJacobs_graded3a}),  one has
 \begin{eqnarray*}
 a* c * a &=&(\sum_{i=1}^{N}a_i)*(\sum_{j=1}^{N}c_j)*(\sum_{k=1}^{N}a_k)\\
          &=& \sum_{j=1}^{N}\sum_{1 \leq i < k \leq N}(a_i*c_j*a_k +  a_k*c_j*a_i) + \sum_{j=1}^{N}\sum_{i=1}^{N}(a_i*c_j*a_i) \\
          &=& 0
  \end{eqnarray*}
 We have shown that $ a* c * a = 0, \forall a, c \in R,$ which by
Theorem \ref{ThmLRI} implies condition \textbf{lri}.

The implication   (2) $\Longrightarrow$ (1)  follows from
\cite[Theorem~7.10]{Tatyana}.
\end{proof}

\section{Constructions and  examples}

 It is not difficult to construct a Jacobson radical ring $R= (R, +, *)$ with $y*x*y\neq 0$, for some $x,y \in R$, for example one can
  use  Golod-Shafarevich theorem.
  Another way to find such radical rings is to fix a field $F$, to consider the free noncommutative $F$-algebra $S$  (without unit) generated by a finite set $X$,  and let
  $I$ be the two-sided ideal
 \[I=S^4=\{\sum_{i=1}^ns_{1,i}*s_{2,i}*s_{3,i}*s_{4,i}\mid s_{1,i},s_{2,i},s_{3,i},s_{4,i}\in S\}.\] Then the quotient
 $R=S/I$ is a nil-algebra ($a^4=0, \forall a \in R$), hence $R$ is a Jacobson radical ring.
 Moreover $x*y*x \neq 0,$ for any $x, y\in X,$ and therefore the corresponding brace $(R, +, \cdot)$ does not satisfy \textbf{lri}.

 Note that Theorem \ref{Thm_main}  provides us with a class of symmetric groups $(G,r)$ and their left braces $(G, +, \cdot)$ each of which
 is not two-sided, but satisfies \textbf{lri} and \textbf{Raut}, e.g. $G = G(X,r)$, where $(X,r)$ is a square-free solution of arbitrary cardinality and $\mpl (X,r) = 2$.

\begin{theorem}\label{6}
Let $F$ be a  field of characteristic two, and let $A$ be the free $F$-algebra (without identity element) generated by
the elements $x, y$.  Let $I$ be
the two-sided ideal of $A$ generated by the set
\[W = \{x*y*y+y*y*x, x*x*y+y*x*x\}\bigcup \{x_{1}*x_{2}*x_{3}*x_{4} \mid x_{1}, x_{2}, x_{3},
x_{4}\in \{x,y\}\},\] and let $R$ be the quotient ring $R=A/I$. Then
$(R, +, *)$ is a graded Jacobson radical ring and the associated
brace $(R, +, \cdot )$ satisfies  condition \textbf{Raut}  but the symmetric group
$(R, r_R)$
does not satisfy \textbf{lri}. Moreover, $\mpl (R, r_R) = 3$
\end{theorem}
\begin{proof} Let $X=\{ x,y\}$. Observe that,  $R= (R, +, *)$ is a graded radical
ring $R=\oplus_{i=1}^{\infty} R_{i}$ with $R_{i}=0$ for every
$i>3$, and $R_1=$ Span$_F X$. By Proposition~\ref{PropGraded_mpl},
$\mpl (R, r_R) = 3$, since $R_{3}\neq 0$.
 It is easy to show that $W$ is a Groebner basis of the ideal $I$ w.r.t.
 the degree-lexicographic order on the free semigroup $\langle x,y\rangle$. (Here the semigroup multiplication is denoted by $*$, and we assume $x > y$).
Hence  the set
\[x,\; y,\; x*x, \; x*y,\; y*x, \;y*y, \; y*y*y,\; y*y*x,\;  y*x*y,\; y*x*x,\; x*y*x,\; x*x*x\]
project to an $F$- basis of $R$, considered as an $F$- vector space. In particular,
 $x*y*x\neq 0, y*x*y\neq 0$ in $R$, hence by Theorem~\ref{ThmLRI}, $R$ doesn't satisfy \textbf{lri}.
We shall show that $R$ satisfies \textbf{Raut}. By Theorem
\ref{ThmE}   it will be enough to show \[a*b*c + c*b*a =0,
\;\;\forall a, b, c \in X.\] Clearly, at least two of the elements
$a,b,c$ coincide. If $a=c$ ($a=b=c$ is also possible) then $a*b*c +
c*b*a = a*b*a +  a*b*a =0$, since the field $F$  has characteristic
$2$. If $a=b\neq c$, then $a*b*c + c*b*a = a*a*c + c*a*a = 0$ holds
in $R$, since by construction the element $a*a*c + c*a*a \in A$ is
contained in the ideal $I$.
 Similarly, if $b = c\neq a$ one has $a*b*c + c*b*a = a*b*b + b*b*a = 0$. We have shown that $a*b*c + c*b*a= 0$ for all $a, b, c \in X,$  therefore
 $R$ satisfies condition \textbf{Raut}.
\end{proof}

\begin{theorem}\label{8}
Let $F$ be a  field of arbitrary characteristic, and let $A$ be the free $F$-algebra (without identity element) generated by
the elements $x, y$.  Let $I$ be
the two-sided ideal of $A$ generated by the set of monomials
\[W =  \{x_{1}*x_{2}*x_{3}*x_{4} \mid x_{1}, x_{2}, x_{3},
x_{4}\in \{x,y\}\},\]
 and let $(R, +, *)$ be the monomial algebra  $R=A/I$.
 Then $(R, +, *)$ is a graded Jacobson radical ring and the associated
brace $(R, +, \cdot )$ does not satisfy  condition \textbf{Raut}. Moreover, $\mpl (R, r_R) = 3$.
\end{theorem}
\begin{proof}
Note first that $I$ is a monomial ideal, generated by the set $W$ of
all monomials of length $4$ in $A$, so $R= (R, +, *)$ is a graded
algebra, moreover as a nil-algebra, $R$ is a graded Jacobson radical
ring  $R=\oplus_{i=1}^{\infty} R_{i}$ with $R_{i}=0$ for every
$i>3$, and $R_1=$ Span$_F \{x,y\}$. By
Proposition~\ref{PropGraded_mpl}, this implies $\mpl (R, r_R) = 3$
(since $R_{3}\neq 0$).
 The set $W$ is a Groebner basis of $I$,
so the set of all words in $x,y$ of length $\leq 3$ projects to an
$F$- basis of $R$. In particular, $x*x*y + y*x*x$ is a nonzero
element of $R$, and setting $a=b=x, c = y$ we see that $a*b*c +
c*b*a= x*x*y + y*x*x \neq 0$ in $R$. It follows from
Theorem~\ref{PropE} that the two-sided brace $(R, +, \cdot)$ does
not satisfy \textbf{Raut}.
\end{proof}

$ $

{\small Ferran Ced{\' o}, Departament de Matem\`{a}tiques,
Universitat Aut\`{o}noma de Barcelona, 08193 Bellaterra (Barcelona),
Spain
\\

Tatiana Gateva-Ivanova,  American University in Bulgaria, 2700
Blagoevgrad,  and Institute of Mathematics and Informatics,
Bulgarian Academy of
Sciences, 1113 Sofia, Bulgaria \\

Agata Smoktunowicz, School of Mathematics, The University of
Edinburgh, James Clerk Maxwell Building, The Kings Buildings,
Mayfield Road EH9 3JZ, Edinburgh}
\\
E-mail: \email{ cedo@mat.uab.cat, Tatyana@aubg.edu,
A.Smoktunowicz@ed.ac.uk}

\end{document}